\documentclass{article}
\usepackage[utf8]{inputenc}
\usepackage{amsmath,amssymb,amsthm}
\usepackage{tabularx,tabulary}
\usepackage[style=numeric,backend=biber]{biblatex}
\usepackage{placeins}
\usepackage{hyperref}
\usepackage{tikz-cd}
\hypersetup{
    colorlinks,
    citecolor=blue,
    filecolor=black,
    linkcolor=blue,
    urlcolor=blue
}

\def\op{\operatorname}
\newcommand{\Z}{\mathbb Z}
\newcommand{\Aut}{\op{Aut}}
\newcommand*{\email}[1]{\href{mailto:#1}{\nolinkurl{#1}}}

\newtheorem{theorem}{Theorem}

\newtheorem{lemma}[theorem]{Lemma}

\newtheorem{algorithm}[theorem]{Algorithm}
\newtheorem{question}[theorem]{Question}

\theoremstyle{definition}
\newtheorem{definition}[theorem]{Definition}

\theoremstyle{remark}
\newtheorem*{remark}{Remark}

\title{Graph coverings and (im)primitive homology: some new examples of exceptionally low degree}
\date{\today}
\author{Destine Lee, Iris Rosenblum-Sellers, Jakwanul Safin, Anda Tenie}

\begin{document}

\maketitle

\begin{abstract}
Given a finite covering of graphs $f : Y \to X$, it is not always the case that $H_1(Y;\mathbb{C})$ is spanned by lifts of primitive elements of $\pi_1(X)$. In this paper, we study graphs for which this is not the case, and we give here the simplest known nontrivial examples of covers with this property, with covering degree as small as 128. Our first step is focusing our attention on the special class of graph covers where the deck group is a finite $p$-group. For such covers, there is a representation-theoretic criterion for identifying deck groups for which there exist covers with the property. We present an algorithm for determining if a finite $p$-group satisfies this criterion that uses only the character table of the group. Finally, we provide a complete census of all finite $p$-groups of rank $\geq 3$ and order $< 1000$ satisfying this criterion, all of which are new examples.
\end{abstract}

\section{Introduction}
Given a graph cover $Y$ of a finite graph $X$ with finite deck group $G$, there is a $G$-action on $H_1(Y;\mathbb{C})$ determined by the induced map on homology of each homeomorphism of $Y$ corresponding to a deck transformation. $H_1(Y;\mathbb{C})$ has then the structure of a $G$-representation. Following the work of Farb and Hensel \cite{MR3576293}, we now define the \textit{primitive homology} of $Y$ as a $G$-subrepresentation of $H_1(Y;\mathbb{C})$ and determine cases where this subrepresentation is strict.

We say an element $x$ of a free group $F_n$ is \textit{primitive} if it is part of a free basis. In other words, for some set $\{x, b_1, b_2,\dots, b_{n-1}\}$ we have $F_n=\langle x, b_1, b_2, \dots, b_{n-1} \rangle$. Equivalently, $x$ is primitive if there exists an automorphism $\alpha: F_n \rightarrow F_n$, such that $\alpha(a)=x$ for some standard generator $a$ of $F_n$. For a finite group $G$, we say an element is \textit{primitive} with respect to a particular surjection $\phi: F_n \rightarrow G$ if it is the image of a primitive element under $\phi$, where $n$ is the size of a minimal generating set for $G$, i.e., the \emph{rank} of $G$. We show in Section \ref*{comp} that for finite $p$-groups (the class of groups under study in this paper) the set of primitive elements is independent of the choice of surjection.

\begin{remark}
We impose the restriction on $n$ because when $n$ exceeds the rank of $G$, we often have a primitive element in the kernel which, as shown by Farb and Hensel \cite{MR3576293}, never leads to examples with the property under study in this paper.
\end{remark}

\begin{definition}[Primitive Homology]
Given a graph cover $Y$ of a finite graph $X$ we define the \emph{primitive homology} of Y, denoted by $H_1^\mathrm{prim}(Y;\mathbb{C}),$ to be the $\mathbb{C}$-span of components coming from lifts of primitive elements of the fundamental group of $X$. 
\end{definition}

Note that this is a $G$-subrepresentation of $H_1(Y;\mathbb{C})$. It is then natural to ask whether this subrepresentation is strict.
\begin{question}
Is $H_1^\mathrm{prim}(Y;\mathbb{C})=H_1(Y;\mathbb{C})$ for every normal cover Y of X?
\end{question}
Farb and Hensel reformulate this question in a purely group theoretic way. Recall the Gasch\"{u}tz formula:
\begin{equation*}
    H_1(Y;\mathbb{C}) \cong \mathbb{C}[G]^{n-1} \oplus \mathbb{C},
\end{equation*}
where $\mathbb{C}[G]$ denotes the regular representation and $n$ is the rank of the free group $\pi_1(X)$. In particular, the $G$-representation corresponding to the action of $G$ on $H_1(Y;\mathbb{C})$ contains every irreducible representation of $G$ as a subrepresentation.

Using this fact, we now define $\op{Irr}^{\op{pr}}(\phi, G),$ an essential piece in translating our topological question into an algebraic one.

\begin{definition}[Primitive Irreducible Representation]
For a finite group $G$ and surjective homomorphism $\phi : F_n \rightarrow G$ (where $n$ is the rank of $G$), the \emph{primitive irreducible representations} $\op{Irr}^{\mathrm{pr}}(\phi, G)$ of $G$ are the irreducible representations $\rho: G \rightarrow \Aut(V)$ such that $\rho(\gamma)(v)=v$ for some primitive $\gamma \in G$ (with respect to $\phi$) and some nonzero vector $v \in V$. Conversely, we say an irreducible representation is imprimitive if it does not belong to $\op{Irr}^{\mathrm{pr}}(\phi, G)$.
\end{definition}
Malestein and Putman \cite{MR4012346} as well as Farb and Hensel \cite{MR3576293} showed that
\[
    H_1^{\mathrm{prim}} (Y;\mathbb{C}) \subseteq \left( \bigoplus_{V_i\in \op{Irr}^{\mathrm{pr}}(\phi, G)} V_i^{(n-1)\op{dim}(V_i)} \right) \oplus \mathbb{C}_{\mathrm{triv}}.
\]
It follows that if there exists an imprimitive irreducible representation of $G$, i.e., if $\op{Irr}^{\text{pr}}(\phi, G) \neq \op{Irr}(G)$, then $H_1^{\text{prim}}(Y;\mathbb{C})$ is a strict subrepresentation of $H_1(Y;\mathbb{C})$.

It is usually the case that $\op{Irr}^{\op{pr}}(\phi, G)=\op{Irr}(G)$ for all surjections $\phi:F_n\to G$, but not always the case; an example of a group without this property for $n=2$ is the group $Q_8$ of unit quaternions. This paper is concerned with finding examples for larger $n$, i.e., groups $G$ of rank $\geq 3$ such that $\op{Irr}^{\op{pr}}(\phi, G) \neq \op{Irr}(G)$ for some surjection $\phi: F_n \rightarrow G$. We say such groups have Property \textit{II}.

\begin{definition}[Property \textit{II}]
    Given a finite group \(G\) and a surjective homomorphism \(\phi:F_n \rightarrow G\), we say G is \emph{imprimitive irreducible} under \(\phi\) (Property \(II_{\phi}\)) if there exists an imprimitive representation of \(G\), that is, a representation of G with the property that $\phi(x)(v)\neq v$ for all primitive $x \in F_n$ and nonzero $v \in V$. Moreover, we say that $G$ has Property \textit{II} if $G$ has Property \textit{II}$_\phi$ for some $\phi$.
\end{definition}

In the discussion to follow, we work in the special case where $G$ is a $p$-group. The implications of this assumption are explored in Section \ref{comp}, where in particular we exploit this assumption to develop an efficient computational approach to finding groups with Property \textit{II}. In Section \ref{imofprim}, we show that one can view the set of primitive elements of a $p$-group as the complement of the Frattini subgroup. In Section \ref{thealg}, we develop an algorithm that uses the character table of a $p$-group to determine this complement. In addition, we show that it can even be determined from the character table whether a $p$-group has Property \textit{II}. In Section \ref{results} we present the groups with Property \textit{II} that our search uncovered. The following theorem summarizes our findings.
\begin{theorem}
There is one group (and only one group) of rank 3 and order 128 with Property $II$. In addition, there are 7 groups of order 256 and rank $\geq 3$ and 165 groups of order 512 and rank $\geq 3$ with Property $II$; all of these groups are of rank 3.
\end{theorem}

\begin{remark}
For context, there are in total 883 groups of rank 3 and order 128, 6,190 of rank 3 and order 256, and 58,859 of rank 3 and order 512.
\end{remark}

Finally, in Section \ref{g128}, we provide a description of the group of order 128, which we denote as $G_{128}$, and demonstrate using the algorithm developed in Section \ref{thealg} that it indeed has Property \textit{II}.

\section{Theoretical Framework} \label{comp}

In \cite{MR3576293}, Farb and Hensel ask which groups $G$ have $\op{Irr}^\text{pr}(\phi, g) = \op{Irr}(G)$ for every surjective homomorphism $\phi : F_n \twoheadrightarrow G$. In the language defined in the introduction, this is equivalent to asking which groups $G$ do not have Property \textit{II}. Farb and Hensel provide a partial answer to this question. Namely, they show that no abelian groups of rank $\geq 2$ have Property \textit{II}, and that no $2$-step nilpotent groups of rank $\geq 3$ have Property \textit{II}. That is, for rank $\geq 3$, there are no abelian or $2$-step nilpotent examples of Property \textit{II}. On the other hand, they were able to produce 2-step nilpotent examples for rank 2. This invites the following question: For rank $\geq 3$, are there examples of Property \textit{II}, and if so, what is the minimal nilpotence degree required to see such examples? In this section, we take a computational approach to answering this question and show this minimal nilpotence degree to be $3$.

For the remainder of our discussion, we assume that $G$ is a finite $p$-group for some prime $p$. As shall be demonstrated, this hypothesis renders the problem especially amenable to computation, and as a matter of fact, whether $G$ has Property \textit{II} may be deduced solely using the character table. Our team implemented an algorithm for doing so, and using it, ran a search through all $p$-groups of order $\leq 1000$ and of rank $\geq 3$; this search yielded 173 groups with Property \textit{II}, all of which were of rank 3, and all of which were $2$-groups. Specifically, these groups were of order 128, 256, and 512, with only one of order 128. In the sequel, we denote this group of order 128 by $G_{128}$.

It is worth emphasizing that our hypothesis is crucial for the performance of the algorithm, for it allows us to circumvent the computationally expensive process of enumerating images of primitive elements. We begin this section by discussing exactly how our hypothesis enables us to do so. After this, we present the algorithm we used, and finally, we provide an overview of the 173 groups found followed by a description of $G_{128}$ as a metabelian extension of $\mathbb{Z}_2 \times \mathbb{Z}_2 \times \mathbb{Z}_2$ by $\mathbb{Z}_4 \times \mathbb{Z}_2 \times \mathbb{Z}_2$.

\subsection{Images of Primitive Elements} \label{imofprim}
We start by recalling a few group-theoretic preliminaries.

\begin{definition}[Frattini Subgroup]
The \emph{Frattini subgroup} $\Phi(G)$ of a group $G$ is the intersection of all maximal proper subgroups of $G$.
\end{definition}

Because any automorphism of $G$ will only permute its maximal subgroups, $\Phi(G)$ is actually a characteristic subgroup of $G$. In particular, $\Phi(G)$ is a normal subgroup, so the quotient group $G / \Phi(G)$ is well-defined. The following result relates generating sets for $G$ to generating sets for $G / \Phi(G)$.

\begin{lemma}[Burnside Basis Theorem]
Let $p$ be a prime and suppose that $G$ is a $p$-group. Then $V = G / \Phi(G)$ is an $\mathbb{F}_p$-vector space, and a set $S \subseteq G$ generates $G$ if and only if its image in $V$ spans $V$. Furthermore, $S$ is a minimal generating set if and only if this image is a basis for $V$.
\end{lemma}

\begin{definition}[Rank of a Finite $p$-group]
Let $p$ be a prime and suppose that $G$ is a finite $p$-group. Then the \emph{rank} of $G$ is the dimension of $V = G / \Phi(G)$ as an $\mathbb{F}_p$-vector space. By the Burnside Basis Theorem, this is also the size of a minimal generating set of $G$.
\end{definition}

Fix a finite $p$-group $G$, and fix any element $g_1 \in G - \Phi(G)$. Then because $V$ is a vector space and $\pi(g_1) \neq 0$, the collection $\{\pi(g_1)\}$ may be extended to a basis for $V$. Let $\{\pi(g_1), v_2, \dots, v_n\}$ be any such basis. By choosing one element $g_i \in G$ for each element $v_i$ such that $\pi(g_i) = v_i$, we obtain, according to the Burnside Basis Theorem, a minimal generating set $S = \{g_1, \dots, g_n\}$ of $G$. This leads us to the following lemma.

\begin{lemma}[Imprimitivity is independent of surjection] \label{byesurj}
Let $p$ be a prime and suppose that $G$ is a $p$-group. Furthermore, put $n$ to be the rank of $G$. Then the subset $S \subseteq G$ of primitive elements is precisely the complement $G - \Phi(G)$ of the Frattini subgroup in $G$.
\end{lemma}

\begin{proof}
Fix any $g_1 \in G - \Phi(G)$. It follows from the preceding discussion that $\{g_1\}$ may be extended to a minimal generating set $\{g_1, \dots, g_n\}$ for $G$.

Choose any free basis $\{a_1, \dots, a_n\}$ for $F_n$ and define $\phi : F_n \rightarrow G$ to be the unique homomorphism mapping $a_i$ to $g_i$ for every $1 \leq i \leq n$. Because the image of $\phi$ contains a generating set, $\phi$ is surjective, and this in turn implies that $g_1$ is the image of $a_1$, a primitive element in $F_n$, under a surjective homomorphism. Since $n$ is the rank of $G$, this proves that $g_1$ is primitive.

Conversely, suppose $g_1 \in \Phi(G)$. Then $\pi(g_1) = 0$ is not an element of any basis for $G / \Phi(G)$, so by the Burnside Basis Theorem, $g_1$ is not contained in any minimal generating set for $G$. Let $n$ be the rank of $G$, and let $\{a_1,\dots,a_n\}$ be any free basis for $F_n$. Then a homomorphism $\phi : F_n \rightarrow G$ is surjective if and only if $\{\phi(a_1),\dots,\phi(a_n)\}$ is a minimal generating set for $G$. Hence, if $\phi$ is surjective, then $g_1$ cannot be the image any $a_i$.
\end{proof}

We have thus shown that, under our hypothesis, the problem of determining the primitive elements of $G$ reduces to taking a set-theoretic complement. Also notice that the set of primitive elements of $G$ is now independent of the choice of surjection $\phi$, i.e., primitivity of an element is now an intrinsic notion. This allows us to check for Property \textit{II}$_\phi$ for all surjections $\phi$ simultaneously. Of course, there now instead arises the problem of actually computing $\Phi(G)$, but when compared with the general case, our current position is advantageous because the Frattini subgroup---a union of conjugacy classes---can be computed from the character table of $G$ alone. The following subsection outlines an algorithm for doing so.

\subsection{The Algorithm}\label{thealg}

In this subsection, we present an algorithm for determining whether a group has Property \textit{II} using little more than the character table of the group. At the heart of this algorithm is the following result.

\begin{lemma} \label{findimprim}
A finite $p$-group $G$ has Property \textit{II} if and only if there exists an irreducible $G$-representation $V$ with character $\chi_V : G \rightarrow \mathbb{C}$ such that for all primitive $g \in G$,
\begin{equation}
    \sum_{i=1}^{\op{ord}(g)}\chi_{V} (g^i)=0.
\end{equation}
\end{lemma}
\begin{proof}
Let $g$ be any primitive element of $G$ and assume that $g$ fixes a nonzero vector $v \in V$. Then every element of $\langle g \rangle$ fixes $v$, and in fact the entire subspace generated by $v$. Conversely, if $\langle g \rangle$ fixes a nontrivial subspace $W \subseteq V$, then in particular, $g$ fixes a nonzero vector. In the language of representation theory, this means that $g$ fixes a nonzero vector $v \in V$ if and only if $\op{Res}_{\langle g \rangle}^G V$ has a trivial subrepresentation, i.e.,
\[
\langle \left. \chi_{V} \right|_{\langle g \rangle} , \chi_{\mathrm{triv}} \rangle = \frac{1}{|\langle g \rangle|}\sum_{i=1}^{\op{ord}(g)}\chi_{V} (g^i) \neq 0, \]
where $\chi_\mathrm{triv}$ denotes the trivial character of $\langle g \rangle$. This proves the assertion.
\end{proof}

One of the important advantages of Equation (1) is that the value of the sum, when considered as a function in $g$, depends only on the conjugacy class to which $g$ belongs. To see why, recall that characters are class functions and that powers of conjugate elements are conjugate, i.e., if $g$ and $h$ are conjugate, then so are $g^k$ and $h^k$. Thus, we may consider the criterion laid out in Lemma \ref*{findimprim} on the level of conjugacy classes. That is, in order to verify whether Equation (1) holds for all primitive $g$ in a given conjugacy class, it suffices to verify the equation for just one primitive representative. 

Conveniently, primitivity is also a property of conjugacy classes; for any given conjugacy class, either every element is primitive or no elements are primitive. This is the content of the following lemma.

\begin{lemma} \label{primconj}
Suppose $G$ is a finitely-generated group, and suppose that $g, h \in G$ are conjugate elements. Then $g$ is primitive if and only if $h$ is primitive.
\end{lemma}

\begin{proof}
Conjugation is an automorphism, and elements of free bases are precisely the images of any other free basis under an automorphism.
\end{proof}

Thus, assuming that it is known which conjugacy classes contain primitive elements, we have the following algorithm.

\begin{algorithm} \label{algorithm}
Let $G$ be a finite $p$-group. The following determines whether $G$ has Property \textit{II}.
\begin{enumerate}
    \item[(1)] Select an irreducible representation $\phi : G \rightarrow \op{Aut}(V)$ and let $\chi$ denote the character of this representation.
    \item[(2)] Select a conjugacy class with a primitive element and select an arbitrary representative $g$. According to Lemma \ref*{primconj}, $g$ is necessarily a primitive element.
    \item[(3)] Verify whether Equation (1) holds for $\chi$ and $g$. If it does not, then return to Step 1, unless there do not remain any irreducible representations, in which case $G$ does not have Property \textit{II}.
    \item[(4)] If there remain conjugacy classes with primitive elements, return to Step 2. Else, $G$ has Property \textit{II}.
\end{enumerate}
\end{algorithm}

Observe that if the character table and power map on the conjugacy classes (i.e., a map $f : C(G) \times \mathbb{Z} \rightarrow C(G)$, where $C(G)$ denotes the collection of conjugacy classes of $G$) of a group $G$ is given, then the algorithm above may be executed without further reference to $G$. This is convenient because both pieces of information are readily available for the $p$-groups in the Small Groups library that is provided by GAP. However, it still remains to determine which conjugacy classes contain primitive elements. Of course, Lemma \ref*{primconj} allows us to select an arbitrary representative of any conjugacy class and invoke Lemma \ref*{byesurj}, but the following lemma allows us to do so again knowing nothing more than the character table of $G$.

\begin{lemma}[Criterion for Primitivity] \label{findprim}
Let $p$ be a prime, $G$ be a $p$-group, and $g \in G$ be an element. Then $g$ is primitive if and only if there exists a linear character $\chi : G \rightarrow \mathbb{C}^\times$ such that
\begin{enumerate}
    \item[(1)] $\chi(g) \neq 1$, and
    \item[(2)] for all $h \in G$, $\op{ord}(\chi(g)) \geq \op{ord}(\chi(h))$.
\end{enumerate}
\end{lemma}

\begin{proof}
We begin by assuming the conditions (1) and (2) hold, and show that $g$ must then be primitive. Let $\chi$ be a linear character satisfying the two conditions and recall from character theory that the commutator subgroup $[G, G]$ of $G$ is the intersection of the kernels of the linear characters of $G$. Thus, (1) implies that $g$ lies in the complement of the commutator subgroup. However, the commutator subgroup is in general a strict subgroup of the Frattini subgroup, and according to Lemma 6, it remains to be shown that $x \not\in \Phi(G) - [G, G]$. This is where (2) comes in.

Suppose for contradiction that $g\in \Phi(G) - [G, G]$, i.e., that $g$ is of the form $xy^p$ for some $x \in [G, G]$ and nontrivial $y \in G$. Then $\chi(x) = 0$ because $x \in [G, G]$, so
\[ \chi(g) = \chi(xy^p) = \chi(x)\chi(y)^p = \chi(y)^p. \]
Since $G$ is a finite $p$-group and $\chi(g) \neq 1$, $\op{ord}(\chi(y)) = p^k$ for some $k > 0$, and $\op{ord}(\chi(g)) = p^{k-1}$. Taking $h$ in (2) to be $y$, we now have a contradiction. This shows that $g \not\in \Phi(G)$, and by Lemma 6, $g$ is primitive.

Conversely, assume that $g_1$ is primitive, and consider the image $\bar{g}_1$ under the canonical projection $\pi : G \rightarrow G / \Phi(G)$. 
By the Burnside Basis Theorem, $G / \Phi(G)$ is an $\mathbb{F}_p$-vector space, and by Lemma 7, $\bar{g}_1$ is not the zero element. Hence, there exist elements $\bar{g}_2, \dots, \bar{g}_n \in G / \Phi(G)$ such that $\bar{g}_1, \bar{g}_2, \dots, \bar{g}_n$ is a basis for $G / \Phi(G)$. This basis provides the following direct sum decomposition of $G / \Phi(G)$.
\[
G / \Phi(G) \cong \langle \bar{g_1} \rangle \oplus \langle \bar{g}_2 \rangle \oplus \dots \oplus \langle \bar{g}_n \rangle.
\]
Let $\chi_0 : G / \Phi(G) \rightarrow \mathbb{C}^\times$ be the any linear character mapping $\bar{g}_1$ to $\omega = e^{2\pi i/p}$. Then $\chi = \chi_0 \circ \pi$ is a linear character on $G$ mapping $g_1$ to $\omega \neq 1$. This character also satisfies the second condition because for any element $h \in G$,
\[ \op{ord}(\chi(h)) = \op{ord}(\chi_0(\bar{h})) \leq p = \op{ord}(\chi_0(\bar{g}_1)) = \op{ord}(\chi(g_1)). \]
This proves the converse.
\end{proof}

Observe that while the above lemma is for a specific element $g$, the criterion itself only involves the values of characters. Thus, Lemma \ref*{findprim} may be used to determine which conjugacy classes consist of primitive elements using only the character table of the group. With this information in hand, it becomes possible to apply Lemma \ref*{findimprim} to search through the irreps of a group for an imprimitive irrep, and this is precisely what Algorithm \ref*{algorithm} accomplishes.

\section{Results}\label{results}

\FloatBarrier

As mentioned earlier, the Small Groups library (smallgrp)---a GAP package---provides a database of lower order $p$-groups. We conducted a thorough search through this database among the groups of rank $\geq 3$ for examples of Property \textit{II}, and were able to find a total of 173. Moreover, all 173 are of nilpotence degree $\geq 3$. Table \ref*{Tab:grouppiichart} provides a breakdown of these 173 groups by order and nilpotence degree.

\begin{table}[ht]
\centering
\begin{tabular}{|l|l|l|l|}
\hline
Order & Total        & 3-step & 4-step \\ \hline
128   & \textbf{1}   & 1      & 0      \\ \hline
256   & \textbf{7}   & 7      & 0      \\ \hline
512   & \textbf{165} & 153    & 12     \\ \hline
\end{tabular}
\caption{Rank $3$ Groups with Property \textit{II} by order and nilpotence degree}
\label{Tab:grouppiichart}
\end{table}

Notice from the data that all groups found were finite $2$-groups, and that 128 is the smallest order for which there exists a rank 3 group with Property \textit{II}. Moreover, there is only one group of order 128. In the sequel, we refer to this group as $G_{128}$. In the section to follow we analyze this group and prove using Algorithm \ref*{algorithm} that $G_{128}$ has Property \textit{II}. More information about the other Property \textit{II} groups that we found, including some groups of rank 2, may be found in the appendix at the end of this document.

\FloatBarrier

\section{The Group of Order 128}\label{g128}

In this section, we provide a brief description of $G_{128}$, the only rank 3 group of order 128 with Property \textit{II}. First of all, $G_{128}$ is 3-step nilpotent, and its Frattini subgroup $\Phi(G_{128})$ is isomorphic to $\Z_4 \times \Z_2 \times \Z_2$, an abelian group. Because $G_{128}$ is of rank 3, we also have that the quotient $G_{128} / \Phi(G_{128})$ is isomorphic to $\Z_2 \times \Z_2 \times \Z_2$. We therefore have the short exact sequence below.
\[
    \begin{tikzcd}
        0 \arrow[r] &
        \Z_4 \times \Z_2 \times \Z_2 \arrow[r] &
        G_{128} \arrow[r] & 
        \Z_2 \times \Z_2 \times \Z_2 \arrow[r] &
        0.
    \end{tikzcd}
\]
This shows $G_{128}$ to be a metabelian extension of $\Z_2 \times \Z_2 \times \Z_2$ by $\Z_4 \times \Z_2 \times \Z_2$. We have verified that this extension is nonsplit. In the display below, we provide one presentation of the group, adopting the convention that $[x,y] = xyx^{-1}y^{-1}$.
\[
G_{128} =
\Biggl\langle 
   \begin{array}{l|cl}
              & a^4 = b^4 = c^2, \\
        a,b,c & [[a, b], b] = [[b, c], c] = a^4[[a, c], a] = 1, \\
              & b^2 = a^4[a, b], a^2 b^2 = [a, c], a^8 = 1  &                                             
    \end{array}
\Biggr\rangle.
\]

In Table \ref*{Tab:chartab}, we provide the character table. As is standard, the irreducible representations of $G_{128}$ are listed vertically in the leftmost column and labeled \texttt{X.1} through \texttt{X.23}. Similarly, the top row of positive integers numbers the conjugacy classes from \texttt{1} to \texttt{23}, and the next row down marks with a \texttt{p} those that are primitive. From each conjugacy class, we have selected one representative and listed these representatives in Table \ref*{Tab:conjreps}.

\begin{table}[ht]
\footnotesize
\begin{verbatim}
         1  2  3  4  5  6  7  8  9  10 11 12 13 14 15 16 17 18 19 20 21 22 23
            p  p  p              p  p  p  p  p  p           p  p  p  p     p

X.1      1  1  1  1  1  1  1  1  1  1  1  1  1  1  1  1  1  1  1  1  1  1  1
X.2      1 -1  1  1  1  1  1  1 -1 -1 -1  1  1  1  1  1  1 -1 -1 -1  1  1 -1
X.3      1  1 -1  1  1  1  1  1 -1  1  1 -1 -1  1  1  1  1 -1 -1  1 -1  1 -1
X.4      1 -1 -1  1  1  1  1  1  1 -1 -1 -1 -1  1  1  1  1  1  1 -1 -1  1  1
X.5      1  1  1 -1  1  1  1  1  1 -1  1 -1  1 -1  1  1  1 -1  1 -1 -1  1 -1
X.6      1 -1  1 -1  1  1  1  1 -1  1 -1 -1  1 -1  1  1  1  1 -1  1 -1  1  1
X.7      1  1 -1 -1  1  1  1  1 -1 -1  1  1 -1 -1  1  1  1  1 -1 -1  1  1  1
X.8      1 -1 -1 -1  1  1  1  1  1  1 -1  1 -1 -1  1  1  1 -1  1  1  1  1 -1
X.9      2  .  .  2 -2  2  2  2  .  .  .  .  . -2 -2 -2  2  .  .  .  . -2  .
X.10     2  .  . -2 -2  2  2  2  .  .  .  .  .  2 -2 -2  2  .  .  .  . -2  .
X.11     2  .  2  .  2 -2  2  2  .  .  .  . -2  . -2  2 -2  .  .  .  . -2  .
X.12     2  . -2  .  2 -2  2  2  .  .  .  .  2  . -2  2 -2  .  .  .  . -2  .
X.13     2  .  .  . -2 -2  2  2  .  .  .  A  .  .  2 -2 -2  .  .  . -A  2  .
X.14     2  .  .  . -2 -2  2  2  .  .  . -A  .  .  2 -2 -2  .  .  .  A  2  .
X.15     2  2  .  .  2  2 -2  2  .  . -2  .  .  .  2 -2 -2  .  .  .  . -2  .
X.16     2 -2  .  .  2  2 -2  2  .  .  2  .  .  .  2 -2 -2  .  .  .  . -2  .
X.17     2  .  .  . -2  2 -2  2  .  A  .  .  .  . -2  2 -2  .  . -A  .  2  .
X.18     2  .  .  . -2  2 -2  2  . -A  .  .  .  . -2  2 -2  .  .  A  .  2  .
X.19     2  .  .  .  2 -2 -2  2  A  .  .  .  .  . -2 -2  2  . -A  .  .  2  .
X.20     2  .  .  .  2 -2 -2  2 -A  .  .  .  .  . -2 -2  2  .  A  .  .  2  .
X.21     2  .  .  . -2 -2 -2  2  .  .  .  .  .  .  2  2  2 -2  .  .  . -2  2
X.22     2  .  .  . -2 -2 -2  2  .  .  .  .  .  .  2  2  2  2  .  .  . -2 -2
X.23*    8  .  .  .  .  .  . -8  .  .  .  .  .  .  .  .  .  .  .  .  .  .  .

A = 2i and . = 0
p = primitive elements
* = imprimitive representation
\end{verbatim}
\caption{Character Table for $G_{128}$}
\label{Tab:chartab}
\end{table}

\FloatBarrier

\begin{table}[ht]
    \centering
    \begin{tabular}{|c|c|c|c|}
        \hline
        Conjugacy Class & Representative & Conjugacy Class & Representative  \\ \hline
        \texttt{1}      & $1$            & \texttt{13}  & $b[a,c]a^4$        \\ \hline
        \texttt{2}      & $a$            & \texttt{14}  & $c[a,b]$           \\ \hline
        \texttt{3}      & $b$            & \texttt{15}  & $[a,b][a,c]a^4$    \\ \hline
        \texttt{4}      & $c$            & \texttt{16}  & $[a,b][b,c]a^4$    \\ \hline
        \texttt{5}      & $[a,b]$        & \texttt{17}  & $[a,c]a^4[b,c]a^4$ \\ \hline
        \texttt{6}      & $[a,c]a^4$     & \texttt{18}  & $abc$              \\ \hline
        \texttt{7}      & $[b,c]a^4$     & \texttt{19}  & $ab[a,c]a^4$       \\ \hline
        \texttt{8}      & $a^4$          & \texttt{20}  & $ac[a,b]$          \\ \hline
        \texttt{9}      & $ab$           & \texttt{21}  & $bc[a,b]$          \\ \hline 
        \texttt{10}     & $ac$           & \texttt{22}  & $[a,b][a,c][b,c]$  \\ \hline
        \texttt{11}     & $a[b,c]a^4$    & \texttt{23}  & $abc[a,b]$         \\ \hline
        \texttt{12}     & $bc$           & ---          & ---                \\ \hline
    \end{tabular}
    \caption{Representatives of the Conjugacy Classes in Table \ref*{Tab:chartab}}
    \label{Tab:conjreps}
\end{table}

\FloatBarrier

By definition, to say that $G_{128}$ has Property \textit{II} is to say that $G$ has an imprimitive irreducible representation. As it turns out, there is exactly one such irrep, and this is the irrep corresponding to the final row, \texttt{X.23}. (This irrep is induced from the alternating irreducible representation of $\langle a^4 \rangle$, a subgroup of order $2$.) Let $\chi_{23} : G_{128} \rightarrow \mathbb{C}$ denote the corresponding character on $G_{128}$. By Lemma \ref*{findimprim}, it suffices to show that:
\begin{equation} \label{g128comp}
    \sum_{i=1}^{\op{ord}(g)}\chi_{23} (g^i)=0,
\end{equation}
for all primitive elements $g$.

For every primitive conjugacy class, one can verify using the provided group presentation that raising any representative to a sufficiently large power yields an element in Conjugacy Class \texttt{8}, the conjugacy class of $a^4$. Since $\chi_{23}$ vanishes on every other conjugacy class except for the conjugacy class of the identity element, the sum on the left-hand side of (\ref*{g128comp}) reduces to
\[ \sum_{i=1}^{\op{ord}(g)}\chi_{23} (g^i) = 0 + \dots + 0 + (-8) + 0 + \dots + 0 + 8=0, \]
for all primitive $g$.
Hence, $G_{128}$ has Property \textit{II}.

For reference, these are the values that $\chi_{23}$ takes on the generators $a$, $b$, and $c$ of $G_{128}$.

\[
\chi_{23}(a) =
\begin{pmatrix}
0 & 1 & 0 & 0  & 0 & 0 & 0  & 0 \\
i & 0 & 0 & 0  & 0 & 0 & 0  & 0 \\
0 & 0 & 0 & 0  & 1 & 0 & 0  & 0 \\
0 & 0 & 0 & 0  & 0 & 1 & 0  & 0 \\
0 & 0 & i & 0  & 0 & 0 & 0  & 0 \\
0 & 0 & 0 & -i & 0 & 0 & 0  & 0 \\
0 & 0 & 0 & 0  & 0 & 0 & 0  & 1 \\
0 & 0 & 0 & 0  & 0 & 0 & -i & 0
\end{pmatrix}.
\]
\[
\chi_{23}(b) =
\begin{pmatrix}
0  & 0  & 1 & 0 & 0  & 0  & 0 & 0 \\
0  & 0  & 0 & 0 & -i & 0  & 0 & 0 \\
-i & 0  & 0 & 0 & 0  & 0  & 0 & 0 \\
0  & 0  & 0 & 0 & 0  & 0  & 1 & 0 \\
0  & -1 & 0 & 0 & 0  & 0  & 0 & 0 \\
0  & 0  & 0 & 0 & 0  & 0  & 0 & i \\
0  & 0  & 0 & i & 0  & 0  & 0 & 0 \\
0  & 0  & 0 & 0 & 0  & -1 & 0 & 0
\end{pmatrix}.
\]
\[
\chi_{23}(c) =
\begin{pmatrix}
0  & 0  & 0  & 1 & 0  & 0 & 0 & 0 \\
0  & 0  & 0  & 0 & 0  & 1 & 0 & 0 \\
0  & 0  & 0  & 0 & 0  & 0 & 1 & 0 \\
-1 & 0  & 0  & 0 & 0  & 0 & 0 & 0 \\
0  & 0  & 0  & 0 & 0  & 0 & 0 & 1 \\
0  & -1 & 0  & 0 & 0  & 0 & 0 & 0 \\
0  & 0  & -1 & 0 & 0  & 0 & 0 & 0 \\
0  & 0  & 0  & 0 & -1 & 0 & 0 & 0 
\end{pmatrix}.
\]

\FloatBarrier

\section{Appendix}

In this section, we record the examples of Property \textit{II} that our search uncovered. Table \ref*{rank3} lists all rank 3 groups with Property \textit{II} of order $\leq 1000$, and Table \ref*{rank2} lists some examples of rank 2. (We emphasize that in contrast to Table \ref*{rank3}, this list is not comprehensive.) In both tables, the groups are sorted by order and referred to by their index in the aforementioned Small Groups library, as indexed in GAP 4.11.0 \cite{GAP4}.

\begin{table}[h]
\caption{Rank 3 Groups with Property \textit{II}}
\begin{tabularx}{1\textwidth}{|c|X|}
\hline
 Order & Index\\
 \hline
 128 & 802  \\
 \hline
 256 & 3592, 3594, 3596, 3598, 4128, 4240, 4300\\
 \hline
512 &  12383, 12384, 12386, 12387, 12390, 12393, 12395, 12396, 12401, 12402, 12403, 12404, 24190, 24224, 24226, 
 24227, 24270, 24272, 24273, 24362, 24363, 24366, 24367, 24370, 24412, 24413, 24416, 24424, 24426, 24427, 
 24430, 24431, 24438, 24440, 24441, 24516, 24518, 24519, 24522, 24523, 24530, 24532, 24533, 24632, 24633, 
 24636, 24642, 24643, 24646, 24724, 24725, 24728, 24734, 24735, 24738, 24823, 24824, 24826, 24880, 24884, 
 24886, 24926, 24930, 24932, 25031, 25033, 25035, 25049, 25051, 25053, 25155, 25157, 25167, 25169, 25271, 
 25273, 25283, 25285, 25375, 25377, 25387, 25389, 25432, 25434, 25444, 25446, 25534, 25536, 25538, 25548, 
 25550, 25677, 25679, 25681, 25691, 25693, 25780, 25782, 25795, 25797, 27373, 27402, 27404, 27432, 27434, 
 27467, 27468, 27469, 27514, 27515, 27516, 27537, 27538, 27539, 27554, 27556, 27581, 27582, 27590, 27592, 
 27593, 27605, 27606, 27628, 27630, 27649, 27650, 27651, 27676, 27677, 27679, 27718, 27719, 27752, 27753, 
 27792, 27793, 27794, 28263, 28296, 28297, 28326, 28327, 28354, 28355, 28439, 28440, 28505, 28506, 28507, 
 30769, 31035, 31198, 53188, 53189, 53190, 53191, 53192, 53193, 53194, 53195, 53198, 53199, 56025, 56474 \\
\hline
\end{tabularx}
\label{rank3}
\end{table}

\begin{table}
\caption{A Sampling of Rank 2 Groups with Property \textit{II}}
\begin{tabularx}{1\textwidth}{|c|X|}
\hline
 Order & Index\\
\hline
8 & 4\\
\hline
16 & 4, 9\\
\hline
32 & 2, 8, 10, 12, 13, 14, 15, 20\\
\hline
64 & 5, 7, 9, 11, 13, 14, 15, 16, 17, 18, 19, 20, 21, 22, 23, 24, 25, 39, 43, 44, 45, 46, 47, 48, 49, 54 \\
\hline
128 & 3, 4, 5, 6, 7, 8, 9, 10, 11, 12, 13, 14, 15, 16, 17, 18, 19, 20, 21, 22, 23, 24, 25, 26, 27, 28, 29, 30, 
 31, 32, 33, 34, 35, 36, 37, 38, 39, 40, 41, 54, 55, 56, 57, 58, 59, 60, 64, 66, 69, 70, 80, 82, 83, 84, 85, 
 86, 90, 94, 95, 96, 97, 98, 99, 100, 101, 102, 103, 104, 105, 106, 107, 108, 109, 110, 111, 112, 113, 114, 
 115, 116, 117, 118, 119, 120, 121, 122, 123, 124, 125, 126, 127, 143, 148, 152, 153, 154, 155, 156, 157, 
 158, 163 \\
\hline
256 & 2, 3, 4, 5, 6, 7, 8, 9, 10, 11, 12, 13, 14, 15, 16, 17, 18, 19, 20, 21, 22, 23, 24, 25, 26, 27, 28, 29, 30, 
 31, 32, 33, 34, 35, 36, 37, 38, 50, 51, 52, 53, 54, 55, 57, 59, 61, 63, 65, 66, 69, 70, 73, 74, 77, 78, 81, 
 83, 85, 86, 87, 88, 89, 96, 101, 106, 107, 108, 109, 110, 111, 112, 113, 114, 115, 116, 117, 118, 119, 120, 
 121, 122, 123, 124, 125, 126, 127, 128, 129, 130, 131, 132, 133, 134, 135, 136, 137, 138, 139, 140, 141, 
 142, 143, 144, 145, 146, 147, 148, 149, 150, 151, 152, 153, 154, 155, 156, 157, 158, 159, 160, 161, 162, 
 163, 164, 165, 166, 167, 168, 169, 170, 171, 172, 173, 174, 175, 176, 177, 178, 179, 180, 181, 182, 183, 
 184, 185, 186, 187, 188, 189, 190, 191, 192, 193, 194, 195, 196, 197, 198, 199, 200, 201, 202, 203, 204, 
 205, 206, 207, 208, 209, 210, 211, 212, 213, 214, 215, 216, 217, 218, 219, 220, 221, 222, 223, 224, 225, 
 226, 227, 228, 229, 230, 231, 232, 233, 234, 235, 236, 237, 238, 239, 240, 241, 242, 243, 244, 245, 246, 
 247, 248, 249, 250, 251, 252, 253, 254, 255, 256, 257, 258, 259, 260, 261, 262, 263, 264, 265, 266, 267, 
 268, 269, 270, 271, 272, 273, 274, 275, 276, 277, 278, 279, 280, 281, 282, 283, 284, 285, 286, 287, 288, 
 289, 290, 291, 292, 293, 294, 295, 296, 297, 298, 299, 300, 301, 302, 303, 304, 305, 306, 307, 308, 309, 
 310, 311, 312, 313, 314, 315, 329, 331, 343, 350, 351, 352, 353, 354, 355, 356, 357, 358, 359, 360, 361, 
 362, 363, 364, 365, 366, 372, 374, 380, 381, 389, 393, 407, 409, 414, 415, 416, 417, 418, 419, 420, 421, 
 422, 423, 424, 425, 426, 427, 428, 429, 430, 431, 433, 437, 438, 439, 440, 441, 442, 443, 444, 445, 446, 
 447, 448, 449, 450, 451, 452, 453, 454, 455, 456, 457, 458, 459, 460, 461, 462, 463, 464, 465, 466, 467, 
 468, 469, 470, 471, 472, 473, 474, 475, 476, 477, 478, 479, 480, 481, 482, 483, 484, 485, 486, 487, 488, 
 489, 490, 491, 492, 493, 494, 495, 496, 526, 530, 531, 532, 533, 534, 535, 536, 541 \\
\hline
512 & Not checked\\
\hline
729 & 6, 8, 10, 11, 12, 14, 15, 18, 21, 26, 52, 53, 54, 55, 78, 79, 80, 81, 82, 83, 96, 101 \\
\hline
\end{tabularx}
\label{rank2}
\end{table}

\FloatBarrier

\section*{Acknowledgements}
The results presented were obtained as part of the Summer 2020 Columbia Math Undergraduate Summer Research Program at Columbia University. First of all, we would like to express gratitude to the Department of Mathematics at Columbia University and Michael Woodbury for organizing the program. In particular, we extend special thanks to our advisors Nick Salter and Maithreya Sitaraman for their thoughtful guidance throughout the research process. More specifically, we thank Nick Salter for sharing with us both his knowledge and passion for the field and his careful comments on the many revisions this document underwent, and we thank Maithreya Sitaraman for the stimulating discussions and ideas he shared with us. Finally, Nick would like to thank Benson Farb, Sebastian Hensel, Justin Malestein, and Andy Putman for some very helpful correspondence on the topic.

\printbibliography

\noindent\emph{E-mail addresses:}
\begin{itemize}
    \item[] DL: \email{dll2141@columbia.edu}
    \item[] IRS: \email{igr2102@columbia.edu}
    \item[] JS: \email{js5142@columbia.edu}
    \item[] AT: \email{ast2175@columbia.edu}
\end{itemize}

\end{document}